\documentclass{amsart}
\usepackage{amsthm}
\usepackage{amssymb}
\usepackage{amsmath}
\usepackage{amsmath,amscd}
\usepackage{color}
\usepackage{hyperref}
\usepackage[all,arc]{xy}

\newtheorem{thm}{Theorem}[section]
\newtheorem{cor}[thm]{Corollary}
\newtheorem{prop}[thm]{Proposition}
\newtheorem{lem}[thm]{Lemma}

\theoremstyle{definition}
\newtheorem{defin}[thm]{Definition}

\newtheorem{examp}[thm]{Example}

\theoremstyle{remark}
\newtheorem{rem}[thm]{Remark}

\newcommand{\nc}{\newcommand}
\nc{\rnc}{\renewcommand}
\nc{\bb}[1]{{\mathbb #1}}
\nc{\bbA}{\bb{A}}\nc{\bbB}{\bb{B}}\nc{\bbC}{\bb{C}}\nc{\bbD}{\bb{D}}
\nc{\bbE}{\bb{E}}\nc{\bbF}{\bb{F}}\nc{\bbG}{\bb{G}}\nc{\bbH}{\bb{H}}
\nc{\bbI}{\bb{I}}\nc{\bbJ}{\bb{J}}\nc{\bbK}{\bb{K}}\nc{\bbL}{\bb{L}}
\nc{\bbM}{\bb{M}}\nc{\bbN}{\bb{N}}\nc{\bbO}{\bb{O}}\nc{\bbP}{\bb{P}}
\nc{\bbQ}{\bb{Q}}\nc{\bbR}{\bb{R}}\nc{\bbS}{\bb{S}}\nc{\bbT}{\bb{T}}
\nc{\bbU}{\bb{U}}\nc{\bbV}{\bb{V}}\nc{\bbW}{\bb{W}}\nc{\bbX}{\bb{X}}
\nc{\bbY}{\bb{Y}}\nc{\bbZ}{\bb{Z}}
\nc{\mbf}[1]{{\mathbf #1}}
\nc{\bfA}{\mbf{A}}\nc{\bfB}{\mbf{B}}\nc{\bfC}{\mbf{C}}\nc{\bfD}{\mbf{D}}
\nc{\bfE}{\mbf{E}}\nc{\bfF}{\mbf{F}}\nc{\bfG}{\mbf{G}}\nc{\bfH}{\mbf{H}}
\nc{\bfI}{\mbf{I}}\nc{\bfJ}{\mbf{J}}\nc{\bfK}{\mbf{K}}\nc{\bfL}{\mbf{L}}
\nc{\bfM}{\mbf{M}}\nc{\bfN}{\mbf{N}}\nc{\bfO}{\mbf{O}}\nc{\bfP}{\mbf{P}}
\nc{\bfQ}{\mbf{Q}}\nc{\bfR}{\mbf{R}}\nc{\bfS}{\mbf{S}}\nc{\bfT}{\mbf{T}}
\nc{\bfU}{\mbf{U}}\nc{\bfV}{\mbf{V}}\nc{\bfW}{\mbf{W}}\nc{\bfX}{\mbf{X}}
\nc{\bfY}{\mbf{Y}}\nc{\bfZ}{\mbf{Z}}
\nc{\bfa}{\mbf{a}}\nc{\bfb}{\mbf{b}}\nc{\bfc}{\mbf{c}}\nc{\bfd}{\mbf{d}}
\nc{\bfe}{\mbf{e}}\nc{\bff}{\mbf{f}}\nc{\bfg}{\mbf{g}}\nc{\bfh}{\mbf{h}}
\nc{\bfi}{\mbf{i}}\nc{\bfj}{\mbf{j}}\nc{\bfk}{\mbf{k}}\nc{\bfl}{\mbf{l}}
\nc{\bfm}{\mbf{m}}\nc{\bfn}{\mbf{n}}\nc{\bfo}{\mbf{o}}\nc{\bfp}{\mbf{p}}
\nc{\bfq}{\mbf{q}}\nc{\bfr}{\mbf{r}}\nc{\bfs}{\mbf{s}}\nc{\bft}{\mbf{t}}
\nc{\bfu}{\mbf{u}}\nc{\bfv}{\mbf{v}}\nc{\bfw}{\mbf{w}}\nc{\bfx}{\mbf{x}}
\nc{\bfy}{\mbf{y}}\nc{\bfz}{\mbf{z}}

\nc{\mcal}[1]{{\mathcal #1}}
\nc{\calA}{\mcal{A}}\nc{\calB}{\mcal{B}}\nc{\calC}{\mcal{C}}\nc{\calD}{\mcal{D}}
\nc{\calE}{\mcal{E}} \nc{\calF}{\mcal{F}}\nc{\calG}{\mcal{G}}\nc{\calH}{\mcal{H}}
\nc{\calI}{\mcal{I}}\nc{\calJ}{\mcal{J}}\nc{\calK}{\mcal{K}}\nc{\calL}{\mcal{L}}
\nc{\calM}{\mcal{M}}\nc{\calN}{\mcal{N}}\nc{\calO}{\mcal{O}}\nc{\calP}{\mcal{P}}
\nc{\calQ}{\mcal{Q}}\nc{\calR}{\mcal{R}}\nc{\calS}{\mcal{S}}\nc{\calT}{\mcal{T}}
\nc{\calU}{\mcal{U}}\nc{\calV}{\mcal{V}}\nc{\calW}{\mcal{W}}\nc{\calX}{\mcal{X}}
\nc{\calY}{\mcal{Y}}\nc{\calZ}{\mcal{Z}}
\nc{\fA}{\frak{A}}\nc{\fB}{\frak{B}}\nc{\fC}{\frak{C}} \nc{\fD}{\frak{D}}
\nc{\fE}{\frak{E}}\nc{\fF}{\frak{F}}\nc{\fG}{\frak{G}}\nc{\fH}{\frak{H}}
\nc{\fI}{\frak{I}}\nc{\fJ}{\frak{J}}\nc{\fK}{\frak{K}}\nc{\fL}{\frak{L}}
\nc{\fM}{\frak{M}}\nc{\fN}{\frak{N}}\nc{\fO}{\frak{O}}\nc{\fP}{\frak{P}}
\nc{\fQ}{\frak{Q}}\nc{\fR}{\frak{R}}\nc{\fS}{\frak{S}}\nc{\fT}{\frak{T}}
\nc{\fU}{\frak{U}}\nc{\fV}{\frak{V}}\nc{\fW}{\frak{W}}\nc{\fX}{\frak{X}}
\nc{\fY}{\frak{Y}}\nc{\fZ}{\frak{Z}}
\nc{\fa}{\frak{a}}\nc{\fb}{\frak{b}}\nc{\fc}{\frak{c}} \nc{\fd}{\frak{d}}
\nc{\fe}{\frak{e}}\nc{\fFf}{\frak{f}}\nc{\fg}{\frak{g}}\nc{\fh}{\frak{h}}
\nc{\fri}{\frak{i}}\nc{\fj}{\frak{j}}\nc{\fk}{\frak{k}}\nc{\fl}{\frak{l}}
\nc{\fm}{\frak{m}}\nc{\fn}{\frak{n}}\nc{\fo}{\frak{o}}\nc{\fp}{\frak{p}}
\nc{\fq}{\frak{q}}\nc{\fr}{\frak{r}}\nc{\fs}{\frak{s}}\nc{\ft}{\frak{t}}
\nc{\fu}{\frak{u}}\nc{\fv}{\frak{v}}\nc{\fw}{\frak{w}}\nc{\fx}{\frak{x}}
\nc{\fy}{\frak{y}}\nc{\fz}{\frak{z}}

\newcommand{\one}{1\hskip-3.5pt1}
\newcommand{\csm}{{c_{\text{SM}}}}
\newcommand{\csmT}{{c^T_{\text{SM}}}}
\newcommand{\ssm}{{s_{\text{SM}}}}
\newcommand{\ssmT}{{s^T_{\text{SM}}}}
\newcommand{\csmTv}{{c^{T,\vee}_{\text{SM}}}}

\DeclareMathOperator{\stab}{stab}

\DeclareMathOperator{\pt}{pt}

\DeclareMathOperator{\End}{End}
\DeclareMathOperator{\Pic}{Pic}
\DeclareMathOperator{\id}{id}
\DeclareMathOperator{\loc}{loc}
\DeclareMathOperator{\Frac}{Frac}
\DeclareMathOperator{\BS}{BS}
\DeclareMathOperator{\SL}{SL}
\DeclareMathOperator{\Gr}{Gr}

\title{Structure constants for Chern classes of Schubert cells}
\author{Changjian Su}
\address{Department of Mathematics, University of Toronto, Toronto, ON, Canada}
\email{changjiansu@gmail.com}

\begin{document}

\begin{abstract}
A formula for the structure constants of the multiplication of Schubert classes is obtained in \cite{GK19}. In this note, we prove analogous formulae for the Chern--Schwartz--MacPherson (CSM) classes and Segre--Schwartz--MacPherson (SSM) classes of Schubert cells in the flag variety. By the equivalence between the CSM classes and the stable basis elements for the cotangent bundle of the flag variety, a formula for the structure constants for the latter is also deduced.
\end{abstract}
\maketitle

\section{Introduction}
In the equivariant cohomology of a flag variety, there is a natural basis given by the fundamental classes of Schubert varieties. It is well known that the structure constants of multiplication of this basis (or its dual basis) enjoy a positivity property \cite{G01}. This also holds in the equivariant K theory of the flag varieties \cite{B02,AGM11,K17}.

However, a manifestly positive formula for the structure constants (and its equivariant K theory analogue) is only completely known for the Grassmannians and 2-step partial flag variety, see \cite{KT03,Bu02,BKPT16,KZJ17}. In certain special cases, these structure constants are just the localizations of the basis elements, which are given by  positive formulae \cite{B99,AJS94}. Recently, a manifestly polynomial formula is found by Goldin and Knutson \cite{GK19} both in the equivariant cohomology and equivariant K theory of the complete flag varieties. 

Knutson communicated to the author that, with Zinn-Justin, they can produce puzzle formulae for 3-step and 4-step partial flag varieties \cite{KZJ}. Instead of considering the Schubert classes in the flag variety, they consider quotient of stable basis elements by the zero section class in the equivariant cohomology of the cotangent bundle of the partial flag variety, which has a natural $\bbC^*$ action by dilating the cotangent fibers. 

The stable basis (or stable envelope) is introduced by Maulik and Okounkov \cite{MO19} in their work on quantum cohomology of Nakajima quiver varieties. Using the stable envelope, they constructed geometrical solutions, called R matrices, of the Yang--Baxter equations. Through the general RTT formalism \cite[Section 5.2]{MO19}, a Yangian can be constructed from these geometric R matrices. By its very definition, the Yangian acts on the equivariant cohomology of the Nakajiama varieties, generalizing earlier constructions of Varagnolo \cite{V00} via correspondences.

The stable basis is not only defined for Nakajima quiver varieties, it is also defined for a large class of varieties called symplectic resolutions, among which the cotangent bundle of the flag variety is the most classical example. The stable basis for the cotangent bundle is studied in \cite{S17, RTV15}. By convolution \cite{CG10} and the result in \cite{L89}, the graded affine Hecke algebra acts on the cohomology of the cotangent bundle of the flag variety, and it is shown in \cite{S17} that the stable basis elements are permuted by the Hecke operators.

The graded affine Hecke algebra also appears in the work of Aluffi and Mihalcea \cite{AM16} on the Chern--Schwartz--MacPherson (CSM) class of Schubert cells. The CSM class theory is a Chern class theory for singular varieties, which was constructed by MacPherson \cite{M74}. In the case of the flag variety, the CSM class of a Schubert cell was conjectured by Aluffi and Mihalcea \cite{AM09,AM16} to be a non-negative linear combination of the Schubert classes. The Grassmannian case is proved by Huh \cite{H16}.

Aluffi and Mihalcea show that the CSM classes of the Schubert cells are also permuted by the Hecke operators \cite{AM16}. Thus, the pullback of the stable basis elements to the flag variety are identified with the CSM classes of the Schubert cells, see \cite{AMSS17,RV15}. Besides, the stable basis elements are related to the characteristic cycles of regular holonomic $\calD$ modules on the flag varieties, which are effective by definition. This observation is used in \cite{AMSS17} to prove the non-equivariant positivity conjecture of Aluffi and Mihalcea for any (partial) flag varieties. We refer interested readers to the survey papers \cite{O15,O18,SZ19,S} for more applications of the stable envelopes to representation theory and enumerative geometry problems. 

Now we have identified the numerator of the classes considered by Knutson and Zinn-Justin \cite{KZJ}, the stable basis elements, with the CSM classes. On the other hand, if we pullback the class of the zero section in the $\bbC^*$-equivariant cohomology of the cotangent bundle, and set the $\bbC^*$-equivariant parameter to 1, we get the total Chern class of the flag variety up to a sign. Thus, the classes considered in \textit{loc. cit.} can be identified with the quotient of the CSM classes of the Schubert cells by the total Chern class of the flag variety, which are called the Segre--Schwartz--MacPherson (SSM) classes of the Schubert cells, see \cite{AMSS19a}. Under the non-degenerate Poincar\'e pairing on the equivariant cohomology of the flag variety, the CSM classes and SSM classes are dual to each other, just as the usual Schubert classes and the opposite ones. 

The main Theorem of this note is a formula for the structure constants of the SSM classes of the Schubert cells. To state it, let us introduce some notation. Let $G$ be a complex Lie group with Borel subgroup $B$ and maximal torus $T$. For any $w$ in the Weyl grop $W$, let $Y(w)^\circ:=B^-wB/B\subset G/B$ be the opposite Schubert cell in the flag variety, where $B^-$ is the opposite Borel subgroup. The SSM classes are denoted by $\ssmT(Y(w)^\circ)$, see Section \ref{sec:Chern}. Let $c_{u,v}^w$ be the structure constants of $\{\ssmT(Y(w)^\circ)|w\in W\}$. For any simple root $\alpha$, let $\partial_\alpha$ denote the following operator on $H_T^*(\pt)=\bbC[\ft]$:
\[\partial_\alpha(f)=\frac{f-s_\alpha(f)}{\alpha},\]
where $s_\alpha(f)$ is the usual Weyl group action on $f\in H_T^*(\pt)$. Let $T^\vee_\alpha:=\partial_\alpha+s_\alpha\in \End_\bbC H_T^*(\pt)$. Extend naturally these operators to the fraction field $\Frac H_T^*(pt)$. The formula is
\begin{thm}\label{thm:main}
For any $u,v,w\in W$, let $Q$ be a reduced word for $w$. Then
\[
c_{u,v}^w=\sum_{\substack{R, S\subset Q,\\ \prod R=u,\prod S=v}}\left(\prod_{q\in Q}\frac{\alpha_q^{[q\in R\cap S]}}{1+\alpha_q}s_q(-T^\vee_q)^{[q\notin R\cup S]}\right)\cdot 1\in \Frac H_T^*(pt),\]
where the exponent $``[\sigma]"$ is 1 if the statement $\sigma$ is true, 0 otherwise.
\end{thm}
This is generalized to the partial flag variety case in Theorem \ref{thm:Pcase}. In the non-equivariant limit, this formula computes the topological Euler characteristic of the intersection of three Schubert cells in general positions, see Equation \eqref{equ:nonequiv}. By \cite[Theorem 1.2]{Sch17}, these non-equivariant limit constants also compute the non-equivariant SSM/CSM classes of Richardson cells in terms of SSM/CSM classes of the Schubert cells. 

Since the CSM classes behave well under pushforward, while the SSM classes behave well under pullback (see Lemma \ref{lem:relations}), the proof of Theorem \ref{thm:main} can not be applied to the CSM classes directly. Nonetheless, 
using the relation between the CSM classes and the SSM classes, we can have a formula for the structure constants for the CSM classes of Schubert cells in the complete flag variety, see Theorem \ref{thm:csm}. By the equivalence between the CSM classes and the stable basis for the cotangent bundle of the flag variety \cite{AMSS17,RV15}, we also get the structure constants for the stable basis, see Theorem \ref{thm:stablecstru}. However, these does not generalize to the partial flag variety case.

Goldin and Knutson \cite{GK19} also have a formula in the equivariant K theory of the flag variety. The K-theoretic generalization of the CSM class (resp. SSM class) is the motivic Chern class (resp. Segre motivic Chern class) \cite{BSY}, and the motivic Chern classes of the Schubert cells are also related to the K theory stable basis elements \cite{O17,SZZ17,AMSS19b,FRW18,SZ19}. Unfortunately, the author fails to generalize Theorem \ref{thm:main} to the case of Segre motivic Chern classes, due to the fact that the Hecke operators in equivariant K theory satisfy a quadratic relation, see Remark \ref{rem:k} for more details. Nonetheless, most of the results in Sections \ref{sec:BS} and \ref{sec:proof} can be naturally extended to the equivariant K theory.

This note is structured as follows. In Section \ref{sec:Chern}, we introduce the CSM/SSM classes and recall their basic properties. In Section \ref{sec:BS}, we focus on Bott--Samelson varieties and relate the CSM classes of Schubert cells to the CSM classes of the cells in the Bott--Samelson varieties. By this relation, the structure constants for the SSM classes of Schubert cells are certain linear combinations of the structure constants for some basis in the equivariant cohomology of the Bott--Samelson varieties, and the main theorem is reduced to Theorem \ref{thm:structure2}, which is proved by induction in Section \ref{sec:proof}. Finally, the result is extended to the parabolic case in Section \ref{sec:P}.

{\em Acknowledgments.} The author thanks A. Knutson and A. Yong for discussions. Special thanks go to L. Mihalcea for providing the proof of Theorem \ref{thm:csm}.

\subsection*{Notation.}
Let $G$ be a complex Lie group, with a Borel subgroup $B$ and a maximal torus $T\subset B$. Let $\ft$ be the Lie algebra of the maximal torus. Then the equivariant cohomology of a point $H_T^*(\pt)$ is identified with $\bbC[\ft]$. Let $R^+$ denote the roots in $B$, and $W$ be the Weyl group with Bruhat order $\leq$ and the longest element $w_0$. For any root $\alpha$, let us use $\alpha>0$ to denote $\alpha\in R^+$. For any $w\in W$, let $X(w)^\circ=BwB/B$ and $Y(w)^\circ=B^-wB/B$ be the Schubert cells in the flag variety $G/B$. Let $X(w)=\overline{X(w)^\circ}$ and $Y(w)=\overline{Y(w)^\circ}$ be the Schubert varieties.

\section{Chern classes of Schubert cells}\label{sec:Chern}
In this section, we first recall the definitions of Chern--Schwartz--MacPherson (CSM) classes and Segre--Schwartz--MacPherson (SSM) classes. Then we consider the case of the flag variety, and recall basic properties of the CSM and SSM classes of the Schubert cells.

\subsection{Preliminaries}
Let us first recall the definition of Chern--Schwartz--MacPherson classes. For any quasi projective variety $X$ over $\bbC$, let $\calF(X)$ denote the group of constructible functions on $X$, i.e., $\calF(X)$ consists of functions $\varphi = \sum_Z c_Z \one_Z$, where the sum is over a finite set of constructible subsets $Z \subset X$, $\one_Z$ is the characteristic function of $Z$ and $c_Z \in \bbZ$ are integers. If $f:Y\rightarrow X$ is a proper morphism, we can define a pushforward $f_*:\calF(Y)\rightarrow \calF(X)$ by setting $f_*(\one_Z)(p)=\chi(f^{-1}(p)\cap Z)$, where $Z\subset Y$ is a locally closed subvariety, $p\in X$, and $\chi$ is the topological Euler characteristic. According to a conjecture attributed to Deligne and Grothendieck, there is a unique natural transformation $c_*: \calF \to H_*$ from the functor $\calF$ of constructible functions on a complex quasi projective variety to the homology functor, such that if $X$ is smooth then $c_*(\one_X)=c(TX)\cap [X]$, where $c(TX)$ is the total Chern class. The naturality of $c_*$ means that it commutes with proper pushforward. This conjecture was proved by MacPherson \cite{M74}. The class $c_*(\one_X)$ for possibly singular $X$ was shown to coincide with a class defined earlier by M.-H.~Schwartz \cite{S65a, S65b}. For any constructible subset $Z\subset X$, we call the class $c_{SM}(Z):=c_*(\one_Z)\in H_*(X)$ the \textit{Chern--Schwartz--MacPherson} (CSM) class of $Z$ in $X$. If $X$ is smooth, we call $\ssm(Z):=\frac{c_*(\one_Z)}{c_*(\one_X)}\in H_*(X)$ the \textit{Segre--Schwartz--MacPherson} (SSM) class of $Z$ in $X$ \footnote{If $X$ is not smooth, we can embed $X$ into a smooth ambient space, and use the total Chern class of the ambient space to define the SSM classes, see \cite{AMSS19a}.}. 

The theory of CSM classes was later extended to the equivariant setting by Ohmoto \cite{O06}. Assume that $X$ has a $T$ action. A group $\calF^T(X)$ of {\em equivariant\/} constructible functions is defined by Ohmoto in \cite[Section 2]{O06}. We recall the main properties that we need:
\begin{enumerate} 
\item If $Z \subseteq X$ is a constructible set which is invariant under the $T$-action, its characteristic function $\one_Z$ is an element of $\calF^T(X)$. We will denote by $\calF_{inv}^{T}(X)$ the subgroup of $\calF^{T}(X)$ consisting of $T$-invariant constructible functions 
on $X$. (The group $\calF^T(X)$ also contains other elements,
  but this will be immaterial for us.)
\item Every proper $T$-equivariant morphism $f: Y \to X$ of algebraic varieties induces a homomorphism $f_*^T: \calF^T(X) \to \calF^T(Y)$. The restriction of $f_*^T$ to $\calF_{inv}^{T}(X)$ coincides with the ordinary push-forward $f_*$ of constructible functions. See \cite[Section 2.6]{O06}.
\end{enumerate}

Ohmoto proves \cite[Theorem 1.1]{O06} that there is an equivariant version of MacPherson transformation 
\[c_*^T: \calF^T(X) \to H_*^T(X)\] 
that satisfies $c_*^T( \one_X) = c^T(TX) \cap [X]_T$ if $X$ is a non-singular variety, and that is functorial with respect to proper push-forwards. The last statement means that for all proper $T$-equivariant morphisms $Y\to X$ the following diagram commutes:
\[ 
\xymatrix{ 
\calF^T(Y) \ar[r]^{c_*^T} \ar[d]_{f_*^T} & H_*^T(Y) \ar[d]^{f_*^T} \\ 
\calF^T(X) \ar[r]^{c_*^T} & H_*^T(X) 
}
\] 

\begin{defin}
Let $Z$ be a $T$-invariant constructible subset of $X$. We denote by $\csmT(Z):=c_*^T(\one_{Z}) {~\in H_*^T(X)}$ the $T$-{\em equivariant Chern--Schwartz--MacPherson (CSM) class\/} of $Z$. If $X$ is smooth, we denote by $\ssmT(Z):=\frac{c_*^T(\one_{Z})}{c^T(\one_X)} {~\in H_*^T(X)}$ the $T$-{\em equivariant Segre--Schwartz--MacPherson (SSM) class\/} of $Z$.
\end{defin}

\subsection{CSM and SSM classes of Schubert cells}
In this section, we recall some basic properties of the CSM and SSM classes of the Schubert cells \cite{AM16,AMSS17}. 

\subsubsection{Schubert classes}
The maximal torus $T$ acts on the flag variety $G/B$ by left multiplication. Since $G/B$ is smooth and projective, we identity the $T$-equivariant homology group $H_*^T(G/B)$ with the $T$-equivariant cohomolgoy $H_T^*(G/B)$, which is a module over $H_T^*(\pt)=\bbC[\ft]$. Therefore, we regard the CSM and SSM classes in $H_T^*(G/B)$. The torus fixed points on $G/B$ are in one-to-one correspondence with the Weyl group. For any $w\in W$, $wB\in G/B$ is the corresponding fixed point. For any $\gamma\in H_T^*(G/B)$, let $\gamma|_w\in H_T^*(\pt)$ denote the restriction of $\gamma$ to the fixed point $wB$.

There is a non-degenerate Poincar\'e pairing $\langle-,-\rangle$ on $H_T^*(G/B)$. A natural basis for $H_T^*(G/B)$ is formed by the Schubert classes $\{[X(w)]|w\in W]\}$, with dual basis the opposite Schubert classes $\{[Y(w)]|w\in W\}$. I.e., $\langle[X(w)],[Y(u)]\rangle=\delta_{w,u}$ for any $w,u\in W$. It is well known that the structure constants for the multiplication of the basis $\{[Y(w)]|w\in W\}$ is non-negative \cite{G01}. However, a manifestly positive formula for the structure constants (and its equivariant K theory analogue) is only known in special cases, see \cite{KT03,Bu02,KZJ17,B02,B99,AJS94}. Nevertheless, a manifestly polynomial formula is obtained in \cite{GK19}.

\subsubsection{Hecke action}\label{sec:hecke}
For any simple root $\alpha_i$, let $P_i$ be the minimal parabolic subgroup containing the Bore subgroup $B$. Let $\pi_i:G/B\rightarrow G/P_i$ denote the projection. Then the divided difference operator is $\partial_i:=\pi_i^*\pi_{i*}\in \End_{H_T^*(\pt)}H^*_T(G/B)$. On $H^*_T(G/B)$, we also have a right Weyl group action  induced by the fibration $G/T\rightarrow G/B$ and the right Weyl group action on $G/T$. For any torus weight $\lambda$, let $\calL_\lambda:=G\times_B\bbC_\lambda\in \Pic_T(G/B)$. Then as operators on $H^*_T(G/B)$, we have \cite{BGG73}
\[s_i=\id+c_1^T(\calL_{\alpha_i})\partial_i.\]
Following \cite{AM16,AMSS17}, define \footnote{These operators are the $\hbar=1$ specialization of $L_i$ and $L^\vee_i$ in \cite[Section 5.2]{AMSS17}.} 
\[\calT_i:=\partial_i-s_i, \textit{\quad and \quad} \calT^\vee_i:=\partial_i+s_i.\]
Then they are adjoint to each other, see \cite[Lemma 5.2]{AMSS17}. I.e., for any $\gamma_1,\gamma_2\in H_T^*(G/B)$, 
$\langle\calT_i(\gamma_1),\gamma_2\rangle=\langle\gamma_1,\calT^\vee_i(\gamma_2)\rangle$. Besides, it is proved in \cite[Proposition 4.1]{AM16} that the $\calT_i$'s satisfy the usual relation in the Weyl group. By adjointness and nondegeneracy of the pairing, the $\calT^\vee_i$'s also satisfy the same relations. Thus, for any $w\in W$, we can form $\calT_w$ and $\calT^\vee_w$. These operators $\calT_w$'s (or $\calT^\vee_w$'s) together with the multiplication by the first Chern classes $c_1^T(\calL_\lambda)$ give an action of the degenerate affine Hecke algebra on $H^*_T(G/B)$.

\subsubsection{CSM and SSM classes}\label{sec:CSMSSM}
For any $w\in W$, the CSM class of the Schubert cell $X(w)^\circ$ is described very easily by the Hecke operators as follows (\cite[Corollary 4.2]{AM16})
\begin{equation}\label{equ:csmhecke}
\csmT(X(w)^\circ)=\calT_{w^{-1}}([X(id)]),
\end{equation}
where $[X(id)]$ is the point class $[B]\in H_T^*(G/B)$. Expand the CSM classes in the Schubert classes
\[\csmT(X(w)^\circ)=\sum_{u\leq w}c^T(u;w)[X(u)]\in H_T^*(G/B),\]
where $c^T(u;w)\in H_T^*(\pt)$. The leading coefficient is $c^T(w;w)=\prod_{\alpha>0,w\alpha<0}(1-w\alpha)$, see \cite[Proposition 6.5]{AM16}. Therefore, the transition matrix between the CSM classes of the Schubert cells and the Schubert varieties are triangular with non-zero diagonals. Thus, $\{\csmT(X(w)^\circ)|w\in W\}$ is a basis for the localized equivariant cohomology $H_T^*(G/B)_{\loc}:=H_T^*(G/B)\otimes_{H_T^*(\pt)}\Frac H_T^*(\pt)$, where $\Frac H_T^*(\pt)$ is the fraction field of $H_T^*(\pt)$. Since the non-equivariant limit of the leading coefficients $c(w;w)=1$, the non-equivariant CSM classes of Schubert cells $\{\csm(X(w)^\circ)|w\in W\}$ forms a basis for the cohomolgoy $H^*(G/B)$. 

The CSM classes of the Schubert cells can be identified with the Maulik and Okounkov's stable basis elements (see Section \ref{sec:stable}) for the cotangent bundle of the flag variety, which are related to representation of the Lie algebra of $G$, see \cite{MO19, S17, AMSS17, SZ19}. This type of relation plays an important role in the proof of the non-equivariant case of the positivity conjecture of Aluffi--Mihalcea. I.e., it is conjectured in \cite{AM16} that \[c^T(u;w)\in \bbZ_{\geq 0}[\alpha|\alpha>0].\] The non-equivariant case is proved in \cite{AMSS17}.
 
Recall the SSM class of a Schubert cell $Y(w)^\circ$ is defined by
\[\ssmT(Y(w)^\circ):=\frac{\csmT(Y(w)^\circ)}{c^T(T(G/B))}\in H_T^*(G/B)_{\loc}.\]
Since $c^T(T(G/B))\cup c^T(T^*(G/B))=\prod_{\alpha>0}(1-\alpha^2)\in H^*(G/B)$, see \cite[Lemma 8.1]{AMSS17}, \[\ssmT(Y(w)^\circ)=\frac{\csmT(Y(w)^\circ)c^T(T^*(G/B))}{\prod_{\alpha>0}(1-\alpha^2)}\in \frac{1}{\prod_{\alpha>0}(1-\alpha^2)}H_T^*(G/B).\]
It is easy to see that $\{\ssmT(Y(w)^\circ)|w\in W\}$ is a basis for the localized equivariant cohomology $H_T^*(G/B)_{\loc}$, and in the non-equivariant cohomology, 
\[\ssm(Y(w)^\circ)=\csm(Y(w)^\circ)\cup c(T^*(G/B))=[Y(w)]+\cdots,\] 
where $\cdots$ denotes some element in $H^{> 2\ell(w)}(G/B)$. Therefore, the non equivariant SSM classes $\{\ssm(Y(w)^\circ)|w\in W\}$ form a basis for $H^*(G/B)$.

By \cite[Theorem 7.3]{AMSS17} (after the specialization $\hbar=1$), we have \footnote{Here we have used the fact $(-1)^{\dim G/B}e^{T\times \bbC^*}(T^*(G/B))|_{\hbar=1}=c^T(T(G/B))$, see the proof of Corollary 7.4 in \textit{loc. cit.}.}
\begin{equation}\label{equ:ssmTvee}
\ssmT(Y(w)^\circ)=\frac{1}{\prod_{\alpha>0} (1+\alpha)}\csmTv(Y(w)^\circ).
\end{equation}
Here $\csmTv(Y(w)^\circ)$ is defined to be $\calT^\vee_{w^{-1}w_0}([Y(w_0)])$, where $[Y(w_0)]$ is the point class $[w_0B]\in H_T^*(G/B)$, see \cite[Definition 5.3]{AMSS17}.

The CSM classes and the SSM classes are dual to each other, i.e., for any $w,u\in W$, we have (\cite[Theorem 9.4]{AMSS17})
\begin{equation}\label{equ:duality}
\langle\csmT(X(w)^\circ), \ssmT(Y(u)^\circ)\rangle=\delta_{w,u}.
\end{equation}

\subsection{Structure constants for the SSM classes}
Define the structure constants for the multiplication of the SSM classes by the following formula
\begin{equation}\label{equ:structure1}
\ssmT(Y(u)^\circ)\cup \ssmT(Y(v)^\circ)=\sum_w c_{u,v}^w\ssmT(Y(w)^\circ)\in H_T^*(G/B)_{\loc},
\end{equation}
where $c_{u,v}^w\in \Frac H_T^*(\pt)$. It is easy to see that $w\geq u,v$ in the above summand, and
\[c_{u,w}^w=\ssmT(Y(u)^\circ)|_w.\]

Fix a reduced word $Q=s_{\alpha_{1}}s_{\alpha_{2}}\cdots s_{\alpha_{l}}$ for $w$. For any subword $R\subset Q$, let $\prod R$ denote the product of the simple reflections in $R$. Since a simple reflection can appear more than twice in the a $Q$, we will use the notation $s_1--$ and $--s_1$ to denote the two different subwords in $Q=s_1s_2s_1$. The localization of the CSM classes is (\cite[Corollary 6.7]{AMSS17} with $\hbar=1$)
\begin{equation}\label{equ:csmlocalization}
\csmT(Y(u)^\circ)|_w=\prod_{\alpha>0,w\alpha>0}(1-w\alpha)\sum_{\substack{R\subset Q,\\\prod R=u}}\prod_{i\in R}(\prod_{j\in Q, j< i}s_j) \alpha_i.
\end{equation}
Therefore, 
\begin{equation}\label{equ:restriction}
c_{u,w}^w=\frac{\csmT(Y(u)^\circ)|_w}{c^T(T(G/B))|_w}=\frac{\sum_{\substack{R\subset Q,\\\prod R=u}}\prod_{i\in R}(\prod_{j\in Q, j< i}s_j) \alpha_i}{\prod_{\alpha>0,w\alpha<0}(1-w\alpha)}.
\end{equation}

By Equation \eqref{equ:duality}, we have
\[c_{u,v}^w=\langle\ssmT(Y(u)^\circ)\cup \ssmT(Y(v)^\circ), \csmT(X(w)^\circ)\rangle.\]
Taking the non-equivariant limit and using \cite[Theorem 1.2]{Sch17}, we get 
\begin{equation}\label{equ:nonequiv}
c_{u,v}^w=\chi\left(Y(u)^\circ\cap gY(v)^\circ\cap hX(w)^\circ\right),
\end{equation}
where $\chi$ denotes the topological Euler characteristic, and $g,h\in G$, such that $Y(u)^\circ$, $gY(v)^\circ$ and $hX(w)^\circ$ intersect transversally. Here we also used the fact that for any constructible function $\varphi$ on $G/B$, $\int_{G/B} c_*(\varphi)=\chi(G/B,\varphi)$, which follows directly from the functoriality of the MacPherson transformation $c_*$ applied to the morphism $G/B\rightarrow \pt$. Besides, in the non-equivariant limit, the left hand side of Equation \eqref{equ:structure1} is the SSM class of a Richardson cell by \cite[Theorem 1.2]{Sch17}. Thus, these constants $c_{u,v}^w$ are the expansion coefficients of the SSM (resp. CSM) classes of the Richardson celles in the SSM (resp. CSM) classes of the Schubert cells.

For any simple root $\alpha$, let $\partial_\alpha$ denote the following operator on $H_T^*(\pt)=\bbC[\ft]$:
\[\partial_\alpha(f)=\frac{f-s_\alpha(f)}{\alpha},\]
where $s_\alpha(f)$ is the usual Weyl group action on $f\in H_T^*(\pt)$. Define $T^\vee_\alpha:=\partial_\alpha+s_\alpha\in \End_\bbC H_T^*(\pt)$. Extend naturally these operators to the fraction field $\Frac H_T^*(pt)$.
 
The main theorem of this note is the following formula for the structure constants $c_{u,v}^w$.
\begin{thm}\label{thm:structure1}
For any $u,v,w\in W$, let $Q$ be a reduced word for $w$. Then
\begin{equation*}
c_{u,v}^w=\sum_{\substack{R, S\subset Q,\\ \prod R=u,\prod S=v}}\left(\prod_{q\in Q}\frac{\alpha_q^{[q\in R\cap S]}}{1+\alpha_q}s_q(-T^\vee_q)^{[q\notin R\cup S]}\right)\cdot 1,
\end{equation*}
where the exponent $``[\sigma]"$ is 1 if the statement $\sigma$ is true, 0 otherwise.
\end{thm}
\begin{rem}
\begin{enumerate}
\item 
It is easy to check that when $v=w$, the above formula is the same as the one in Equation \eqref{equ:restriction}.
\item 
In \cite{GK19}, the authors also obtain formulae for the K theory structure constants. The analogue of the CSM/SSM classes in K theory are the motivic Chern classes and Segre motivic Chern classes, see \cite{BSY, AMSS19b}. It is interesting to generalize the above formula for the Segre motivic Chern classes.
\end{enumerate}
\end{rem}
\begin{examp}
Let $G=\SL(3,\bbC)$ with simple roots $\alpha_1,\alpha_2$. Let $\alpha_3:=\alpha_1+\alpha_2$ be the non-simple root. Consider the case $w=s_1s_2s_1$, $u=s_1$ and $v=s_2$. Let $Q=s_1s_2s_1$. Then $R$ can be $s_1--$ or $--s_1$, while $S=-s_2-$. The $c_{u,v}^w$ is computed as follows
\begin{align*}
c_{u,v}^w=&\left(\frac{1}{1+\alpha_1}s_1\frac{1}{1+\alpha_2}s_2\frac{1}{1+\alpha_1}s_1(-T^\vee_1)\right)\cdot 1\\
&+\left(\frac{1}{1+\alpha_1}s_1(-T^\vee_1)\frac{1}{1+\alpha_2}s_2\frac{1}{1+\alpha_1}s_1\right)\cdot 1\\
=&-\frac{1}{(1+\alpha_1)(1+\alpha_2)(1+\alpha_3)}\\
&-\frac{1}{(1+\alpha_1)(1+\alpha_2)(1+\alpha_3)}\\
=&-\frac{2}{(1+\alpha_1)(1+\alpha_2)(1+\alpha_3)}.
\end{align*} 
\end{examp}

\subsection{Structure constants for the CSM classes}
In this section, we give a formula for the structure constants for the CSM classes.

Define the structure constants $d_{u,v}^w$ by 
\begin{equation}\label{equ:structure2}
\csmT(Y(u)^\circ)\cup \csmT(Y(v)^\circ)=\sum_w d_{u,v}^w\csmT(Y(w)^\circ)\in H_T^*(G/B)_{loc},
\end{equation}
where $d_{u,v}^w\in \Frac H_T^*(\pt)$. In particular, $d_{u,w}^w=\csmT(Y(u)^\circ)|_w$.

Define a $\bbC$-linear map $\varphi:H_T^*(G/B)\rightarrow H_T^*(G/B)$ by $\varphi(\gamma)=(-1)^i\gamma$ for any $\gamma\in H^{2i}_T(G/B)$. This is an algebra automorphism of $H_T^*(G/B)$. By definition, for any degree $i$ homogeneous polynomial $f\in H_T^{2i}(\pt)$ on $\ft$, $\varphi(f)=(-1)^if$. Extend $\varphi$ to $\Frac H_T^*(\pt)$ by $\varphi(\frac{f}{g})=\frac{\varphi(f)}{\varphi(g)}$ for any $f,g\in H_T^*(\pt)$. Thus, we can also extend $\varphi$ to $H_T^*(G/B)_{loc}$. 

Then the structure constants $d_{u,v}^w$ is related to the $c_{u,v}^w$ in Equation \eqref{equ:structure1} as follows.
\begin{thm}\label{thm:csm}
For any $u,v,w\in W$, we have
\[d_{u,v}^w=(-1)^{\ell(u)+\ell(v)-\ell(w)}\varphi(c_{u,v}^w)\prod_{\alpha>0}(1-\alpha)\in \Frac H_T^*(\pt).\]
In particular, in the non-equivariant case, $d_{u,v}^w=(-1)^{\ell(u)+\ell(v)-\ell(w)}c_{u,v}^w$.
\end{thm}
\begin{rem}
Using Equations \eqref{equ:csmlocalization}, \eqref{equ:restriction} and the equality \[\prod_{\alpha>0,w\alpha>0}(1-w\alpha)\prod_{\alpha>0,w\alpha<0}(1+w\alpha)=\prod_{\alpha>0}(1-\alpha),\] it is easy to check directly the Theorem when $v=w$.
\end{rem}
\begin{proof}\footnote{The author thanks L. Mihalcea for pointing out this proof.}
By \cite[Proposition 5.4]{AMSS17} (after setting $\hbar=1$),
\[\csmTv(Y(w)^\circ)=(-1)^{\ell(w)}\sum_k(-1)^{\dim G/B-k}\csmT(Y(w)^\circ)_k,\]
where $\csmT(Y(w)^\circ)_k\in H_{2k}^T(X)=H_T^{2\dim G/B-2k}(G/B)$ is the degree $2\dim G/B-2k$ component of $\csmT(Y(w)^\circ)$ defined by 
\[\csmT(Y(w)^\circ)=\sum_k \csmT(Y(w)^\circ)_k.\]
Therefore, applying $\varphi$ to Equation \eqref{equ:ssmTvee}, we get
\[\varphi(\ssmT(Y(w)^\circ))=\frac{(-1)^{\ell(w)}}{\prod_{\alpha>0} (1-\alpha)}\csmT(Y(w)^\circ).\]
Applying the automorphism $\varphi$ to Equation \eqref{equ:structure1} and comparing with Equation \eqref{equ:structure2}, we get
\[d_{u,v}^w=(-1)^{\ell(u)+\ell(v)-\ell(w)}\varphi(c_{u,v}^w)\prod_{\alpha>0}(1-\alpha).\]
\end{proof}

\subsection{Structure constants for the stable basis for $T^*(G/B)$}
\label{sec:stable}
In this section, we give a formula for the structure constants for the stable basis in the cotangent bundle of the complete flag variety $G/B$. 

Let us first recall the definition of the stable basis. The torus $\bbC^*$ acts on $T^*(G/B)$ by scaling the cotangent fiber by a character of $-\hbar$, and it acts trivially on the zero section $G/B$. The fixed points $(T^*(G/B))^{T\times \bbC^*}$ are in one-to-one correspondence with the Weyl group, and they all lie in the zero section. For any $w\in W$, $(wB,0)$ is the corresponding fixed point. For any $\gamma\in H_{T\times \bbC^*}^*(T^*(G/B))$, let $\gamma|_w\in H_{T\times \bbC^*}^*(\pt)=\bbC[\ft][\hbar]$ denote the pullback of $\gamma$ to the fixed point $(wB,0)$. On $ H_{T\times \bbC^*}^*(T^*(G/B))$, there is a non-degenerate Poincar\'e pairing $\langle-,-\rangle_{T^*(G/B)}$ defined by localization as follows
\[\langle\gamma_1,\gamma_2\rangle_{T^*(G/B)}=\sum_w\frac{\gamma_1|_w\gamma_2|_w}{\prod_{\alpha>0}(-w\alpha)(w\alpha-\hbar)}\in \Frac H_{T\times \bbC^*}^*(\pt),\]
where $\gamma_1,\gamma_2\in H_{T\times \bbC^*}^*(T^*(G/B))$. For any $w\in W$, let $T_{Y(w)^\circ}^*(G/B)$ denote the conormal bundle of the opposite Schubert cell $Y(w)^\circ$ inside $G/B$. Then the stable basis is given by
\begin{thm}[\cite{MO19, S17}]
There exist unique $T\times \bbC^*$-equivariant Lagrangian cycles $\{\stab_-(w)\,|\, w\in W\}$ in $T^*(G/B)$ which satisfy the following properties:
\begin{enumerate}
\item $supp (\stab_-(w))\subset \bigcup_{u\geq w} \overline{T_{Y(u)^\circ}^*(G/B)}$;
\item $\stab_-(w)|_w= \prod\limits_{\alpha>0,w\alpha>0}(w\alpha-\hbar)\prod\limits_{\alpha>0,w\alpha<0}w\alpha$;
\item $\stab_-(w)|_u$ is divisible by $\hbar$, for any $u>w$ in the Bruhat order.
\end{enumerate}
\end{thm}
From the first and second properties, it follows that $\{\stab_-(w)\,|\, w\in W\}$ is a basis, which is called the \textit{stable basis}, for the localized equivariant cohomology $H_{T\times \bbC^*}^*(T^*(G/B))_{\loc}$. Here the negative sign $-$ denotes the anti-dominant Weyl chamber. For the dominant Weyl chamber $+$, there is also a stable basis $\{\stab_+(w)\,|\, w\in W\}$, where $\stab_+(w)$ is supported on 
$\bigcup_{u\leq w} \overline{T_{X(u)^\circ}^*(G/B)}$. Moreover, these two bases are dual to each other, i.e., for any $w,u\in W$, $\langle\stab_-(w),\stab_+(u)\rangle_{T^*(G/B)}=(-1)^{\dim G/B}\delta_{w,u}$, see \cite[Remark 2.2(3)]{S17}.

Define the structure constants $e_{u,v}^w$ by the formula
\begin{equation}\label{equ:stablestructre}
\stab_-(u)\cup \stab_-(v)=\sum_w e_{u,v}^w\stab_-(w).
\end{equation}
By the duality, 
\[e_{u,v}^w=(-1)^{\dim G/B}\langle\stab_-(u)\cup \stab_-(v),\stab_+(w)\rangle_{T^*(G/B)}\in \Frac H_{T\times\bbC^*}^*(\pt).\]
By the second property of the stable basis, the intersection of the support for $\stab_-(u)$ and $\stab_+(w)$ is proper. Therefore, $e_{u,v}^w=\langle\stab_-(u)\cup \stab_-(v),\stab_+(w)\rangle_{T^*(G/B)}$ lies in the non-localized cohomology ring $H_{T\times\bbC^*}^*(\pt)$. Since the stable basis elements are given by Lagrangian cycles, each of them lives in the cohomology degree $2\dim G/B$. Thus, a degree count shows $e_{u,v}^w\in H_{T\times\bbC^*}^*(\pt)=\bbC[\ft][\hbar]$ is a homogeneous polynomial of degree $2\dim X$. Therefore, we can recover the homogeneous polynomial $e_{u,v}^w$ from its specialization $e_{u,v}^w|_{\hbar=1}\in H_T^*(\pt)$, which will be related to the structure constants $d_{u,v}^w$ in Equation \eqref{equ:structure2}.

Let $\iota:G/B\hookrightarrow T^*(G/B)$ denote the inclusion of the zero section. By \cite[Proposition 6.9(ii)]{AMSS17}, we have
\begin{equation}\label{equ:stabecsm}
\iota^*(\stab_-(w))|_{\hbar=1}=(-1)^{\dim G/B}\csmT(Y(w)^\circ).
\end{equation}
Therefore, applying $\iota^*$ to Equation \eqref{equ:stablestructre}, letting $\hbar=1$ and comparing with Equation \eqref{equ:structure2}, we get
\begin{thm}\label{thm:stablecstru}
For any $u,v,w\in W$, we have
\[e_{u,v}^w|_{\hbar=1}=(-1)^{\dim G/B}d_{u,v}^w.\]
\end{thm}
However, using the current method, we can not obtain similar results for the cotangent bundle of partial flag varieties. The case of $T^*\Gr(k,n)$ is studied in \cite{C17}.

\section{Bott--Samelson varieties}\label{sec:BS}
The proof of the main Theorem \ref{thm:structure1} uses the Bott--Samelson variety. In this section, we introduce this variety and its basic properties, and reduce the main theorem to a formula for the structure constants for some basis in the equivariant cohomology of the Bott--Samelson variety. In the remaining parts of this note, we fix a word $Q=s_{\alpha_{1}}s_{\alpha_{2}}\cdots s_{\alpha_{l}}$ (not necessarily reduced) in simple reflections.

\subsection{Definition and properties}
Recall that the Bott--Samelson variety associated to the word $Q$ is 
\[\BS^Q:=P_{\alpha_{1}}\times_B P_{\alpha_{2}}\times_B\cdots \times_B P_{\alpha_{l}}/B,\]
where $P_{\alpha_{k}}$ are the minimal parabolic subgroups corresponding to $\alpha_{k}$, and the quotient is by the equivalence relation given by $(g_1,g_2,\cdots g_l)\sim (g_1b_1,b_1^{-1}g_2b_2,\cdots, b_{l-1}^{-1}g_lb_l)$ for any $g_k\in P_{k}$ and $b_k\in B$. Let $[g_1,g_2,\cdots g_l]$ denote the resulting equivalence class in $\BS^Q$.

The maximal torus $T$ acts by left multiplication on $\BS^Q$, and there are $2^{\#Q}$ fixed points. To be more specific, the set of sequences $\{(g_1,g_2,\cdots g_l)\in P_{\alpha_{1}}\times P_{\alpha_{2}}\times \cdots \times P_{\alpha_{k}}|g_{k}\in \{1,s_{k}\}\}$ map bijectively to the fixed points set $(\BS^Q)^T$. Thus, we can index the fixed points by subwords $J\subset\{s_{1},s_{2},\cdots s_{l}\}$ of $Q$.  For a subword $J\subset Q$, the torus weights of the tangent space $T_J\BS^Q$ is (\cite[Lemma 1(iii)]{M07})
\[\{(\prod_{j\in J,j\leq i}s_j) (-\alpha_i)|1\leq i\leq l\}.\]

The Bott--Samelson variety $\BS^Q$ has a cell decomposition as follows, see \cite[Section 3.2]{M07}. For any subword $J\subset Q$, we have a submanifold
\[\BS^J:=\{[g_1,g_2,\cdots g_l]\in \BS^Q| g_k\in B \textit{ if } k\notin J\}\subset \BS^Q,\]
which can be identified with the Bott--Samelson variety for the word $J$. It has an open dense cell 
\[\BS^J_\circ:=\BS^J\setminus \cup_{S\subsetneq J}\BS^S,\]
which contains a unique fixed point corresponding to the subword $J\subset Q$. From definition, it is easy to see that
\begin{equation}\label{equ:cell}
\BS^J=\sqcup_{S\subset J}\BS_\circ^S.
\end{equation}

The equivariant cohomology $H_T^*(\BS^Q)$ has a natural basis given by the fundamental classes of the sub Bott--Samelson varieties $\{[\BS^J]|J\subset Q\}$. Let $\langle-,-\rangle_{\BS^Q}$ denote the non-degenerate Poincar\'e pairing on $H_T^*(\BS^Q)$. By the  Atiyah--Bott localization theorem and the above description of the weights at torus fixed point, we have the following formula for the pairing.
\begin{lem}
For any $\gamma_1,\gamma_2\in H_T^*(\BS^Q)$, we have
\[\langle\gamma_1,\gamma_2\rangle_{\BS^Q}=\sum_{J\subset Q}\frac{\gamma_1|_J \gamma_2|_J}{\prod_{i\in Q}(\prod_{j\in J,j\leq i}s_j) (-\alpha_i)},\]
where $\gamma_i|_J$ denotes the restriction of $\gamma_i$ to the fixed point $J$.
\end{lem}

Let $\pi:\BS^Q\rightarrow G/B$ be the natural map
\[\pi([g_1,g_2,\cdots g_l])=(\prod_ig_i)B/B.\]
The image is the Schubert variety $X(\tilde{\prod} Q)$, where $\tilde{\prod} Q\in W$ is the Demazure product, see \cite{GK19}. If $Q$ is reduced, then $\tilde{\prod} Q=\prod Q\in W$, $\BS^Q$ is a resolution of singularity for $X(\prod Q)$ and $\pi_*([\BS^Q])=[X(\prod Q)]$. Otherwise, $\pi_*([\BS^Q])=0$.

\subsection{CSM classes of cells and dual basis}
For each subword $J\subset Q$, we have the equivariant CSM classes $\csmT(\BS^J_\circ)\in H_T^*(\BS^Q)$. By the same reason as in Section \ref{sec:CSMSSM}, $\{\csmT(\BS^J_\circ)|J\subset Q\}$ is a basis for the localized equivariant cohomology $H_T^*(\BS^Q)_{\loc}$. Using the non-degenerate pairing $\langle-,-\rangle_{\BS^Q}$, we can get a dual basis, denoted by $\{T_J|J\subset Q\}$. I.e., for any two subwords $J,R\subset Q$, 
\begin{equation}\label{equ:duality2}
\langle\csmT(\BS^J_\circ),T_R\rangle_{\BS^Q}=\delta_{R,J}.
\end{equation}
Since the restriction $\csmT(\BS^J_\circ)|_S=0$ if $S\nsubseteq J$, $T_R|_S=0$ if $R\nsubseteq S$.

By the cell decomposition in Equation \eqref{equ:cell}, we get 
\[\csmT(\BS^J)=\sum_{S\subset J}\csmT(\BS^S_\circ).\]
Hence,
\begin{equation}\label{equ:closure}
\langle\csmT(\BS^J),T_S\rangle_{\BS^Q}=\left\{\begin{array}{cc}
1 &, \textit{ if } S\subset J\\
0 &, \textit{ otherwise }
\end{array}\right.
\end{equation}
Moreover, we have the following localization formula for $\csmT(\BS^J)$.
\begin{lem}\label{lem:cellrestriction}
For any $R\subset J$, we have
\[\csmT(\BS^J)|_R=\prod_{j\in J}\left(1+(\prod_{\substack{r\in R\\ r\leq j}}s_r)(-\alpha_j)\right)\prod_{q\in Q\setminus J}(\prod_{\substack{r\in R\\ r\leq q}}s_r)(-\alpha_q).\]
\end{lem}
\begin{proof}
Let $\iota$ denote the inclusion of $\BS^J$ into $\BS^Q$, which is a proper map. Then by the functoriality and normalization properties of the MacPherson transformation $c_*^T$, we get
\[\csmT(\BS^J)=\iota_{*}(\csmT(\BS^J)_J)=\iota_{*}(c^T(T(\BS^J)),\]
Here, $\csmT(\BS^J)_J$ denotes the CSM classes of $\BS^J$ in $H^*_T(\BS^J)$, which equals $c^T(T(\BS^J)$ by the normalization condition. Let $N_{\BS^J/\BS^Q}$ denote the normal bundle of $\BS^J$ inside $\BS^Q$. Then
\begin{align*}
\csmT(\BS^J)|_R&=c^T(T_R(\BS^J))e^T(N_{\BS^J/\BS^Q}|_R)\\
&=\prod_{j\in J}\left(1+(\prod_{\substack{r\in R\\ r\leq j}}s_r)(-\alpha_j)\right)\prod_{q\in Q\setminus J}(\prod_{\substack{r\in R\\ r\leq q}}s_r)(-\alpha_q),
\end{align*}
where $e^T(N_{\BS^J/\BS^Q}|_R)$ denotes the $T$-equivariant Euler class of the fiber at $R$ of the normal bundle $N_{\BS^J/\BS^Q}$.
\end{proof}

These CSM classes of cells in $\BS^Q$ and the dual classes are related to the CSM and SSM classes of Schubert cells by the following lemma.
\begin{lem}\label{lem:relations}
For any subword $J\subset Q$, we have
\[\pi_*(\csmT(\BS^J_\circ))=\csmT(X(\prod J)^\circ),\]
and 
\[\pi^*(\ssmT(Y(w)^\circ))=\sum_{J\subset Q, \prod J=w} T_J.\]
\end{lem} 
\begin{rem}\label{rem:k}
\begin{enumerate}
\item 
The CSM classes behaves well with respect to proper pushforward, while the SSM classes behaves well with respect to pullback. If we switch the position between the CSM and SSM classes, such formulae would not hold. This is the reason why the method of Goldin and Knutson \cite{GK19} can only be applied to the SSM classes, not the CSM classes.
\item 
These formulae can be generalized to equivariant K theory, with the CSM (resp. SSM) classes replaced by the motivic Chern (resp. Segre motivic Chern) classes, see \cite{AMSS19b,MW}. However, since the corresponding Hecke operators in the equivariant K theory satisfy the quadratic relation $(T_s+1)(T_s-q)=0$ instead of the simple relation $T_s^2=1$, the formulae would be much more complicated. For example, there will many terms in first pushforward formula. This prevents the author from computing the structure constants for the Segre motivic Chern classes for the Schubert cells.
\end{enumerate}
\end{rem}
\begin{proof}
Let $J=s_{\alpha_{i_1}}s_{\alpha_{i_2}}\cdots s_{\alpha_{i_k}}\subset Q$.  
By \cite[Lemma 3.1 and Theorem 3.3]{AM16}, we get 
\[\pi_*(\csmT(\BS^J_\circ))=\calT_{i_k}\Bigg(\pi_*\bigg(\csmT(\BS^{s_{\alpha_{i_1}}s_{\alpha_{i_2}}\cdots s_{\alpha_{i_{k-1}}}}_\circ)\bigg)\Bigg).\]
Hence,
\[\pi_*(\csmT(\BS^J_\circ))=\calT_{i_k}\circ\cdots \circ\calT_{i_2}\circ\calT_{i_1}([X(\id)])=\calT_{(\prod J)^{-1}}([X(\id)]),\]
where in the second equality, we used the fact that the operators $\calT_i$'s defined in Section \ref{sec:hecke} satisfy the usual Weyl group relation, see \cite[Proposition 4.1]{AM16}. 
Then the first statement follows from Equation \eqref{equ:csmhecke}.

For the second one, we have
\begin{align*}
\pi^*(\ssmT(Y(w)^\circ))&=\sum_{J\subset Q}\langle\pi^*(\ssmT(Y(w)^\circ)),\csmT(\BS^J_\circ)\rangle_{\BS^Q}T_J\\
&=\sum_{J\subset Q}\langle \ssmT(Y(w)^\circ),\pi_*(\csmT(\BS^J_\circ))\rangle T_J\\
&=\sum_{J\subset Q}\langle \ssmT(Y(w)^\circ),\csmT(X(\prod J)^\circ)\rangle T_J\\
&=\sum_{J\subset Q, \prod J=w} T_J.
\end{align*}
Here the first and the last equalities follows from Equations \eqref{equ:duality2} and \eqref{equ:duality}, the second one follows from the projection formula, and the third one follows from the first equality in this lemma.
\end{proof}

\subsection{Structure constants for the dual basis}
Define the structure constants for the dual basis $\{T_J|J\subset Q\}$ in $H_T^*(\BS^Q)_{\loc}$ by 
\begin{equation}\label{equ:defstructure2}
T_R T_S=\sum_{J}b_{R,S}^J T_J\in H_T^*(\BS^Q)_{\loc},
\end{equation}
where $b_{R,S}^J\in \Frac H_T^*(\pt)$ and $R,S\subset J$. Taking the coefficients of $T_R$ on both sides, we get $T_S|_R=b_{R,S}^R$, which is given in Proposition \ref{prop:restriction}.

These structure constants are related to those $c_{u,v}^w$ in Equation \eqref{equ:structure1} by the following lemma.
\begin{lem}
Assume $Q$ is a reduced word, with product $\prod Q=w\in W$. Then
\[c_{u,v}^w=\sum_{\substack{R,S\subset Q,\\ \prod R=u,\prod S=v}}b_{R,S}^Q.\]
\end{lem}
\begin{proof}
Applying $\pi^*$ to Equation \eqref{equ:structure1} and using Lemma \ref{lem:relations}, we get
\begin{align*}
&\pi^*(\ssmT(Y(u)^\circ))\cup \pi^*(\ssmT(Y(v)^\circ))\\
=&(\sum_{R\subset Q, \prod R=u} T_R)(\sum_{S\subset Q, \prod S=v} T_S)\\
=&\sum_{\substack{R\subset Q, \prod R=u\\S\subset Q, \prod S=v}}\sum_Jb_{R,S}^J T_J\\
=&\sum_z c_{u,v}^z\pi^*(\ssmT(Y(z)^\circ))\\
=&\sum_z c_{u,v}^z \sum_{J\subset Q, \prod S=z} T_J.
\end{align*}
Taking the coefficient of $T_Q$ on both sides, we get
\[c_{u,v}^w=\sum_{\substack{R,S\subset Q,\\ \prod R=u,\prod S=v}}b_{R,S}^Q.\]
\end{proof}
Therefore, the main Theorem \ref{thm:structure1} is reduced to the following
\begin{thm}\label{thm:structure2}
For any word $Q$ (not necessarily reduced) with subwords $R,S\subset Q$, we have
\[b_{R,S}^Q=\left(\prod_{q\in Q}\frac{\alpha_q^{[q\in R\cap S]}}{1+\alpha_q}s_q(-T^\vee_q)^{[q\notin R\cup S]}\right)\cdot 1\in \Frac H_T^*(pt).\]
\end{thm}

\section{Proof of Theorem \ref{thm:structure2}}\label{sec:proof}
In this section, we will first give a localization formula for the dual basis $\{T_J|J\subset Q\}$. Theorem \ref{thm:structure2} will follow by an induction argument on the number of elements in the word $Q$. 
\subsection{Localization of the dual basis}
First of all, we have
\begin{lem}\label{lem:pairingwithclosure}
For any $\gamma\in H_T^*(\BS^Q)$,
\[\langle\gamma, \csmT(\BS^J)\rangle_{\BS^Q}=\sum_{R\subset J}\gamma|_R\prod_{j\in J}\frac{1+(\prod_{\substack{r\in R\\ r\leq j}}s_r)(-\alpha_j)}{(\prod_{\substack{r\in R\\ r\leq j}}s_r)(-\alpha_j)}.\]
\end{lem}
\begin{proof}
By the Atiyah--Bott localization formula and Lemma \ref{lem:cellrestriction}, we have
\begin{align*}
\langle\gamma, \csmT(\BS^J)\rangle_{\BS^Q}&=\sum_{R\subset Q}\frac{\gamma|_R \csmT(\BS^J)|_R}{e^T(T_R\BS^Q)}\\
&=\sum_{R\subset J}\gamma|_R\prod_{j\in J}\frac{1+(\prod_{\substack{r\in R\\ r\leq j}}s_r)(-\alpha_j)}{(\prod_{\substack{r\in R\\ r\leq j}}s_r)(-\alpha_j)}.
\end{align*}
\end{proof}

The localizaiton of the dual basis is given by the following formula.
\begin{prop}\label{prop:restriction}
For any subwords $S\subset R\subset Q$, 
\[T_S|_R=\left(\prod_{r\in R}\frac{\alpha_r^{[r\in S]}}{1+\alpha_r}s_r\right)\cdot 1\in \Frac H_T^*(pt).\]
\end{prop}
\begin{proof}
By the localization theorem, let us denote the classes $T_S'\in H_T^*(\BS^Q)$ defined by the localization conditions 
\[T'_S|_R=\left\{\begin{array}{cc}
\left(\prod_{r\in R}\frac{\alpha_r^{[r\in S]}}{1+\alpha_r}s_r\right)\cdot 1 & \textit{ if }S\subset R\\
0 & \textit{ otherwise }.
\end{array}\right.\]
By Equation \eqref{equ:closure}, we only need to show
\[\langle T'_S,\csmT(\BS^J)\rangle_{\BS^Q}=\left\{\begin{array}{cc}
1, & \textit{ if } S\subset J,\\
0, &\textit{ otherwise }.
\end{array}\right.\]
Since 
\[T'_S|_R=\left(\prod_{r\in R}\frac{\alpha_r^{[r\in S]}}{1+\alpha_r}s_r\right)\cdot 1= \frac{\prod_{s\in S}(\prod_{\substack{r\in R\\ r\leq s}}s_r)(-\alpha_s)}{\prod_{r\in R}\left(1+(\prod_{\substack{r'\in R\\ r'\leq r}}s_{r'})(-\alpha_r)\right)},\] 
Lemma \ref{lem:pairingwithclosure} gives, 
\begin{align*}
&\langle T'_S,\csmT(\BS^J)\rangle_{\BS^Q}\\
=&\sum_{S\subset R\subset J}T_S'|_R\prod_{j\in J}\frac{1+(\prod_{\substack{r\in R\\ r\leq j}}s_r)(-\alpha_j)}{(\prod_{\substack{r\in R\\ r\leq j}}s_r)(-\alpha_j)}\\
=&\sum_{S\subset R\subset J}\frac{\prod_{j\in J\setminus R}\left(1+(\prod_{\substack{r\in R\\ r\leq j}}s_r)(-\alpha_j)\right)}{\prod_{j\in J\setminus S}(\prod_{\substack{r\in R\\ r\leq j}}s_r)(-\alpha_j)}.
\end{align*}
Thus, it suffices to check that for any $S\subset J\subset Q$,
\begin{equation}\label{equ:reducedstep}
\sum_{S\subset R\subset J}\frac{\prod_{j\in J\setminus R}\left(1+(\prod_{\substack{r\in R\\ r\leq j}}s_r)(-\alpha_j)\right)}{\prod_{j\in J\setminus S}(\prod_{\substack{r\in R\\ r\leq j}}s_r)(-\alpha_j)}=1.
\end{equation}
We prove it by induction on the number of words in $J$. If $J=\emptyset$, it is obvious. Now assume $\#J\geq 1$, and Equation \eqref{equ:reducedstep} is true for any $J'$ such that $\#J'<\#J$. Suppose $s_\alpha$ is the first word in $J$ and $J=s_\alpha J_0$. We consider the following two cases according to whether $S$ contains $s_\alpha$ as its first word or not.

{\em Case I:} $S$ contains $s_\alpha$ as the first word. Then for all the $R$ such that $S\subset R\subset J$, $s_\alpha$ will also be the first word for $R$. Let $S=s_\alpha S_0$ and $R=s_\alpha R_0$. Then 
\begin{align*}
&\sum_{S\subset R\subset J}\frac{\prod_{j\in J\setminus R}\left(1+(\prod_{\substack{r\in R\\ r\leq j}}s_r)(-\alpha_j)\right)}{\prod_{j\in J\setminus S}(\prod_{\substack{r\in R\\ r\leq j}}s_r)(-\alpha_j)}\\
=&s_\alpha\left(\sum_{S_0\subset R_0\subset J_0}\frac{\prod_{j\in J_0\setminus R_0}\left(1+(\prod_{\substack{r\in R_0\\ r\leq j}}s_r)(-\alpha_j)\right)}{\prod_{j\in J_0\setminus S_0}(\prod_{\substack{r\in R_0\\ r\leq j}}s_r)(-\alpha_j)}\right)\\
=&s_\alpha(1)\\
=&1,
\end{align*}
where the second equality follows from induction step.

{\em Case II:} $s_\alpha$ is not the first word in $S$. We separate the sum in Equation \eqref{equ:reducedstep} into two terms as follows
\begin{align*}
&\sum_{S\subset R\subset J}\frac{\prod_{j\in J\setminus R}\left(1+(\prod_{\substack{r\in R\\ r\leq j}}s_r)(-\alpha_j)\right)}{\prod_{j\in J\setminus S}(\prod_{\substack{r\in R\\ r\leq j}}s_r)(-\alpha_j)}\\
=&\sum_{\substack{S\subset R\subset J\\ s_\alpha\in R}}\frac{\prod_{j\in J\setminus R}\left(1+(\prod_{\substack{r\in R\\ r\leq j}}s_r)(-\alpha_j)\right)}{\prod_{j\in J\setminus S}(\prod_{\substack{r\in R\\ r\leq j}}s_r)(-\alpha_j)}+\sum_{\substack{S\subset R\subset J\\ s_\alpha\notin R}}\frac{\prod_{j\in J\setminus R}\left(1+(\prod_{\substack{r\in R\\ r\leq j}}s_r)(-\alpha_j)\right)}{\prod_{j\in J\setminus S}(\prod_{\substack{r\in R\\ r\leq j}}s_r)(-\alpha_j)}\\
=&\frac{1}{\alpha}s_\alpha(1)+\frac{1-\alpha}{-\alpha}\\
=&1,
\end{align*} 
where the second equality follows from induction. This finishes the proof of Equation \eqref{equ:reducedstep}.
\end{proof}
Suppose $s_\alpha$ is the first word in $R$ and $R=s_\alpha R_0$. Then an immediate corollary is
\begin{cor}\label{cor:reflections} 
\begin{enumerate}
\item 
If $s_\alpha$ is the first word in $S$ and $S=s_\alpha S_0$, then
\[T_S|_R=\frac{\alpha}{1+\alpha}s_\alpha (T_{S_0}|_{R_0}).\]
\item 
If $s_\alpha$ is not the first word in $S$, then
\[T_S|_R=\frac{1}{1+\alpha}s_\alpha (T_{S}|_{R_0}).\]
\end{enumerate}

\end{cor}

\subsection{Proof of Theorem \ref{thm:structure2}}
In this section, we prove Theorem \ref{thm:structure2} by induction on $\#Q$. 

Restricting both sides of Equation \eqref{equ:defstructure2} to the fixed point $Q$, we get
\[T_R|_Q T_S|_Q=\sum_{R,S\subset J\subset Q}b_{R,S}^J T_J|_Q.\]
Therefore,
\begin{equation}\label{equ:recursion}
b_{R,S}^Q=\frac{T_R|_Q T_S|_Q-\sum_{R,S\subset J\subsetneq Q}b_{R,S}^J T_J|_Q}{T_Q|_Q}.
\end{equation}

\noindent\textit{Proof of Theorem \ref{thm:structure2}}
We prove the Theorem by induction on $\#Q$. If $Q=\emptyset$, it is obvious. Assume the theorem is true for any $Q$ with $\#Q=k\geq 0$. Now assume $\#Q=k+1$, the first word is $s_\alpha$ and $Q=s_\alpha Q_0$. Consider different cases for $R,S$ containing, or not containing the first letter of $Q$.

\textit{Case I}: Suppose neither $R$ or $S$ contains the first word $s_\alpha$ of $Q$. Then we have
\begin{lem}
The sum in Equation \eqref{equ:recursion} can be rewritten as  
\[\sum_{R,S\subset J\subsetneq Q}b_{R,S}^J T_J|_Q=b_{R,S}^{Q_0}\frac{1}{1+\alpha}s_\alpha
(T_{Q_0}|_{Q_0})+\sum_{R,S\subset J\subsetneq Q_0}\frac{1}{(1+\alpha)^2}s_\alpha(b_{R,S}^J)s_\alpha(T_{J}|_{Q_0}).\]
\end{lem}
\begin{proof}
We separate the left hand side into two parts as follows
\begin{align*}
&\sum_{R,S\subset J\subsetneq Q}b_{R,S}^J T_J|_Q\\
=&\sum_{\substack{R,S\subset J\subsetneq Q\\ s_\alpha\notin J}}b_{R,S}^J T_J|_Q+\sum_{\substack{R,S\subset J\subsetneq Q\\ s_\alpha\in J}}b_{R,S}^J T_J|_Q\\
=&\sum_{R,S\subset J\subset Q_0}b_{R,S}^J T_J|_{s_\alpha Q_0}+\sum_{R,S\subset J_0\subsetneq Q_0}b_{R,S}^{s_\alpha J_0} T_{s_\alpha J_0}|_{s_\alpha Q_0}\\
=&\left(b_{R,S}^{Q_0} \frac{1}{1+\alpha}s_\alpha(T_{Q_0}|_{Q_0})+\frac{1}{1+\alpha}\sum_{R,S\subset J\subsetneq Q_0}b_{R,S}^J s_\alpha(T_J|_{ Q_0})\right)\\
&+\left(\frac{\alpha}{1+\alpha}\sum_{R,S\subset J_0\subsetneq Q_0}\frac{1}{1+\alpha}s_\alpha(-T_\alpha^\vee)(b_{R,S}^{J_0}) s_\alpha(T_{J_0}|_{Q_0})\right)\\
=&b_{R,S}^{Q_0}\frac{1}{1+\alpha}s_\alpha
(T_{Q_0}|_{Q_0})+\sum_{R,S\subset J\subsetneq Q_0}\frac{1}{(1+\alpha)^2}s_\alpha(b_{R,S}^J)s_\alpha(T_{J}|_{Q_0}).
\end{align*}
Here the third equality follows from Corollary \ref{cor:reflections} and the induction step for $s_\alpha J_0\subsetneq Q$, and the last equality follows from the formula $s_\alpha(-T_\alpha^\vee)=s_\alpha(-\partial_\alpha-s_\alpha)=-s_\alpha \frac{1}{\alpha}(1-s_\alpha)-1=\frac{1}{\alpha}s_\alpha-\frac{1+\alpha}{\alpha}$.
\end{proof}
Plugging this into Equation \eqref{equ:recursion}, we get
\begin{align*}
b_{R,S}^Q&=\frac{T_R|_Q T_S|_Q-\sum_{R,S\subset J\subsetneq Q}b_{R,S}^J T_J|_Q}{T_Q|_Q}\\
&=\frac{\frac{1}{(1+\alpha)^2}s_\alpha(T_R|_{Q_0} T_S|_{Q_0})-b_{R,S}^{Q_0}\frac{1}{1+\alpha}s_\alpha(T_{Q_0}|_{Q_0})-\sum_{R,S\subset J\subsetneq Q_0}\frac{1}{(1+\alpha)^2}s_\alpha(b_{R,S}^J)s_\alpha(T_{J}|_{Q_0})}{\frac{\alpha}{1+\alpha}s_\alpha(T_{Q_0}|_{Q_0})}\\
&=\frac{1}{\alpha(1+\alpha)}s_\alpha\left(\frac{T_R|_{Q_0} T_S|_{Q_0}-\sum_{R,S\subset J\subsetneq Q_0}b_{R,S}^JT_{J}|_{Q_0})}{T_{Q_0}|_{Q_0}}\right)-\frac{1}{\alpha}b_{R,S}^{Q_0}\\
&=\frac{1}{\alpha(1+\alpha)}s_\alpha(b_{R,S}^{Q_0})-\frac{1}{\alpha}b_{R,S}^{Q_0}\\
&=\frac{1}{1+\alpha}s_\alpha(-T^\vee_\alpha)(b_{R,S}^{Q_0}),
\end{align*}
as desired.

\textit{Case II}: Suppose both $R$ or $S$ contain the first word $s_\alpha$ of $Q$. Let $R=s_\alpha R_0$ and $S=S_\alpha S_0$. Every $J$ satisfying $R,S\subset J\subset Q$ also contains $s_\alpha$ as its first word. Let $J=s_\alpha J_0$. Then by induction step for $J\subsetneq Q$ and Corollary \ref{cor:reflections}, we have
\begin{align*}
\sum_{R,S\subset J\subsetneq Q}b_{R,S}^J T_J|_Q&=\sum_{R_0,S_0\subset J_0\subsetneq Q_0}b_{s_\alpha R_0,s_\alpha S_0}^{s_\alpha J_0} T_{s_\alpha J_0}|_{s_\alpha Q_0}\\
&=\sum_{R_0,S_0\subset J_0\subsetneq Q_0}\frac{\alpha}{1+\alpha}s_\alpha(b_{R_0, S_0}^{J_0})\frac{\alpha}{1+\alpha} s_\alpha(T_{ J_0}|_{Q_0}).
\end{align*}
Putting this into Equation \eqref{equ:recursion}, we get
\begin{align*}
b_{R,S}^Q&=\frac{T_R|_Q T_S|_Q-\sum_{R,S\subset J\subsetneq Q}b_{R,S}^J T_J|_Q}{T_Q|_Q}\\
&=\frac{\alpha}{1+\alpha}s_\alpha\left(\frac{T_{R_0}|_{Q_0} T_{S_0}|_{Q_0}-\sum_{R_0,S_0\subset J_0\subsetneq Q_0}b_{R_0,S_0}^{J_0} T_{J_0}|_{Q_0}}{T_{Q_0}|_{Q_0}}\right)\\
&=\frac{\alpha}{1+\alpha}s_\alpha(b_{R,S}^{Q_0}),
\end{align*}
as desired.

As $b_{R,S}^Q=b_{S,R}^Q$, we are left with the last case.

\textit{Case III}:  Suppose $R=s_\alpha R_0$, while $S$ does not begin with $s_\alpha$. Every $J$ satisfying $R,S\subset J\subset Q$ also contains $s_\alpha$ as its first word. Let $J=s_\alpha J_0$.  Then by induction step for $J\subsetneq Q$ and Corollary \ref{cor:reflections}, we have
\begin{align*}
\sum_{R,S\subset J\subsetneq Q}b_{R,S}^J T_J|_Q&=\sum_{R_0,S\subset J_0\subsetneq Q_0}b_{s_\alpha R_0,S}^{s_\alpha J_0} T_{s_\alpha J_0}|_{s_\alpha Q_0}\\
&=\sum_{R_0,S\subset J_0\subsetneq Q_0}\frac{1}{1+\alpha}s_\alpha(b_{R_0, S}^{J_0})\frac{\alpha}{1+\alpha} s_\alpha(T_{ J_0}|_{Q_0}).
\end{align*}
Plugging this into Equation \eqref{equ:recursion} and using Corollary \ref{cor:reflections}, we get
\begin{align*}
b_{R,S}^Q&=\frac{T_R|_Q T_S|_Q-\sum_{R,S\subset J\subsetneq Q}b_{R,S}^J T_J|_Q}{T_Q|_Q}\\
&=\frac{1}{1+\alpha}s_\alpha\left(\frac{T_{R_0}|_{Q_0} T_{S}|_{Q_0}-\sum_{R_0,S\subset J_0\subsetneq Q_0}b_{R_0,S}^{J_0} T_{J_0}|_{Q_0}}{T_{Q_0}|_{Q_0}}\right)\\
&=\frac{1}{1+\alpha}s_\alpha(b_{R,S}^{Q_0}).
\end{align*}
This finishes the proof in all cases.

\section{Parabolic case}\label{sec:P}
In this section, we generalize Theorem \ref{thm:structure1} to the parabolic case. 

Let $P$ be a parabolic subgroup containing the Borel subgroup $B$. Let $W_P$ be the Weyl group of a Levi subgroup of $P$, and $W^P=W/W_P$ be the set of minimal length representatives. For any $w\in W^P$, let $X(wW_P)^\circ$ (resp. $Y(wW_P)^\circ$) denote the Schubert cell $BwP/P$ (resp. $B^{-}wP/P$). Let $p:G/B\rightarrow G/B$ be the natural projection.

The properties of the CSM/SSM classes of $G/P$ are given by
\begin{lem}\label{lem:BP}
\begin{enumerate}
\item 
For any $w\in W$, 
\[p_*(\csmT(X(w)^\circ))=\csmT(X(wW_P)^\circ)\in H^*_T(G/P).\]
\item 
For any $w\in W^P$,
\[p^*(\ssmT(Y(wW_P)^\circ))=\sum_{x\in W_P}\ssmT(Y(wx)^\circ)\in H^*_T(G/B).\]
\item 
For any $u,w\in W^P$,
\[\langle\csmT(X(wW_P)^\circ),\ssmT(Y(uW_P)^\circ)\rangle=\delta_{w,u}.\]
\end{enumerate}
\end{lem}
\begin{proof}
The first one follows from the functoriality of the MacPherson transformation. The second one follows from the Verdier--Riemann--Roch Theorem (see \cite[Theorem 9.2, or Equation (37)]{AMSS17}). The last one is \cite[Theorem 9.4]{AMSS17}.
\end{proof}

Define the structure constants for the SSM classes as follow:
\[\ssmT(Y(uW_P)^\circ)\cup\ssmT(Y(vW_P)^\circ)=\sum_{w\in W^P}c_{u,v}^w(P)\ssmT(Y(wW_P)^\circ)\in H_T^*(G/P)_{\loc},\]
where $u,v\in W^P$ and $c_{u,v}^w(P)\in \Frac H_T^*(\pt)$. Applying $p^*$ to the above equation and using Lemma \ref{lem:BP}, we get
\[(\sum_{x\in W_P} \ssmT(Y(ux)^\circ))\cup (\sum_{y\in W_P} \ssmT(Y(vy)^\circ))=\sum_{w\in W^P}c_{u,v}^w(P)\sum_{x\in W_P} \ssmT(Y(wx)^\circ).\]
For any $w\in W^P$, taking the coefficient of $\ssmT(Y(w)^\circ)$ on both sides, we obtain
\begin{thm}\label{thm:Pcase}
For any $u,v,w\in W^P$,
\[c_{u,v}^w(P)=\sum_{x\in W_P,y\in W_P}c_{ux,vy}^w.\]
\end{thm}
\begin{rem}
It was communicated by A. Knutson \cite{KZJ} to the author that the structure constants for the SSM classes of the Schubert cells in a partial flag variety can be used to deduce some puzzle formulae for the Schubert structure constants in 2/3/4-step partial flag varieties in type A, generalizing the earlier works \cite{KT03,Bu02,BKPT16}. In fact, the authors in \cite{KZJ} consider the quotient of the stable basis elements \cite{MO19} by the class of the zero section in the equivariant cohomology of the cotangent bundle of the flag variety, which has a natural $\bbC^*$ action by scaling the cotangent fibers. The pullback of the stable basis elements to the flag variety are the CSM classes of the Schubert cells, see Equation \eqref{equ:stabecsm}. Besides, the pullback of the $\bbC^*$-equvivariant class of the zero section is the total Chern class of the flag variety (after setting the $\bbC^*$-equvivariant parameter to 1 and up to a sign). Thus, the classes considered in \cite{KZJ} are precisely the SSM classes of the Schubert cells. This is one of my motivations for this work.
\end{rem}

\bibliographystyle{halpha}

\end{document}